\documentclass[12pt,a4paper,reqno]{amsart}
\usepackage[utf8]{inputenc}
\usepackage{color}
\usepackage{enumerate}
\usepackage[margin=2.5cm]{geometry}
\usepackage{graphicx} 
\usepackage{overpic}
\usepackage{amsaddr}

\newtheorem{lema}{Lemma}[section]
\newtheorem{teo}[lema]{Theorem}
\newtheorem{pro}[lema]{Proposition}
\newtheorem{remark}[lema]{Remark}

\title[Kolmogorov system]{Crossing limit cycles in Piecewise Smooth Kolmogorov Systems: An application to Palomba's model}

\author[Y. R. Carvalho] {Yagor Romano Carvalho$^{1}$}
\address{$^1$ Mathematics Department, Universidade de São Paulo, São Carlos, Brazil}
\email{$^1$yagor.carvalho@usp.br}

\author[L. F. S Golveia] {Luiz F.S. Gouveia$^*$}
\address{$^*$ UNESP, São José do Rio Preto, Brazil/ UNICAMP, Campinas, Brazil}
\email{$^*$fernando.gouveia@unesp.br}

\author[Oleg Makarenkov] {Oleg Makarenkov$^2$}
\address{$^{2}$ University of Texas at Dallas, Dallas, USA}
\email{$^2$ makarenkov@utdallas.edu}



\begin{document}
	
	\maketitle

	\begin{abstract} 
	In this paper, we study the number of isolated crossing periodic orbits, so-called crossing limit cycles, for a class of piecewise smooth Kolmogorov systems defined in two zones separated by a straight line. In particular, we study the number of crossing limit cycles of small amplitude. They are all nested and surround one equilibrium point or a sliding segment.  We denote by $\mathcal M_{K}^{p}(n)$ the maximum number of crossing limit cycles bifurcating from the equilibrium point via a degenerate Hopf bifurcation for a piecewise smooth Kolmogorov systems of degree $n=m+1$. We make a progress towards the determination of the lower bounds $M_K^p(n)$ of crossing limit cycles bifurcating from the equilibrium point via a degenerate Hopf bifurcation for a piecewise smooth Kolmogorov system of degree $n$. Specifically, we shot that $M_{K}^{p}(2)\geq 1$, $M_{K}^{p}(3)\geq 12$, and $M_{K}^{p}(4)\geq 18$. In particular, we show at least one crossing limit cycle in Palomba's economics model, considering it from a piecewise smooth point of view. To our knowledge, these are the best quotes of limit cycles for piecewise smooth polynomial Kolmogorov systems in the literature.
\end{abstract}


\section{Introduction}
The interest in the study of planar piecewise smooth systems is motivated by rich applications in the modeling of real-world phenomena. Development of the theory of piecewise smooth systems was pioneered by Andronov \cite{Andronov1966}.  Nowadays, the theory of piecewise smooth differential equations are used in various applications such as engineering, economy, physics, medicine, and biology, see \cite{AcaBonBro2011, Coombes, Ber2008,HenMicDale1997,Narendra2014,Ogata1990}. 
In this subject, the most classical situation studied in the plane is when two vector fields are defined in two half-planes separated by a straight line. As in the classical qualitative theory of polynomial systems, the location and the number of the isolated crossing periodic orbits, also called crossing limit cycles, have received special attention and can be seen as a natural extension of the 16th Hilbert problem for planar piecewise smooth polynomial vector fields, see for example \cite{ Gianna}. We note that particular cases from Hilbert's 16th problem help us better understand how the characteristics of polynomial systems interfere with the appearance of limit cycles.  This paper is the first work in the literature to show lower bounds for small limit cycles considering Kolmogorv piecewise systems. 
 This kind of systems can have two types of limit cycles: sliding limit cycles or crossing limit cycles. The first ones contains some segment of the line of discontinuity, and the second ones only contain some points of the line of discontinuity. More details on piecewise smooth differential systems can be found in \cite{MakLam2012}. For the polynomial planar case,  for example, piecewise smooth differential systems can be described as
\begin{equation}\label{eqpie}
	\begin{aligned}
		\begin{array}{l}
			\dot{x}=P^{+}(x,y,\lambda^{+}), \; \; \mbox{for} \; y \geq 0, \\
			\dot{y}=Q^{+}(x,y,\lambda^{+}),
		\end{array}  \; \; \; \; \; \; \;  \; \; 
		\begin{array}{l}
			\dot{x}=P^{-}(x,y,\lambda^{-}), \; \; \mbox{for} \; y \leq 0, \\
			\dot{y}=Q^{-}(x,y,\lambda^{-}),
		\end{array}  \\
	\end{aligned}
\end{equation}
where $P^{\pm}$ and $Q^{\pm}$ are polynomials of degree $n$ in the variables $(x,y)$, and $(\lambda^{+},\lambda^{-} )\in \mathbb{R}^{2n^2+6n+4}$ are perturbative parameters. Moreover,  the discontinuity manifold is given by $\Sigma= \{ (x,y)  \in \mathbb{R}^{2} \; | \; y=0\}$ that divides the plane into two half-plane. To define the trajectories on $\Sigma$, we will consider the Filippov convention \cite{Fil1988}.



In this paper,  we are interested in studying the number of small crossing limit cycles bifurcating from the origin in the class of piecewise smooth polynomial differential systems \eqref{eqpie} defined in two zones separated by a straight line where in each zone we have a Kolmogorov system, in which the origin is a monodronic equilibrium point for both Kolmogorov systems. Such a class of systems was proposed by Kolmogorov in \cite{Kolmogorov1936} and is a generalization of the well-known Lotka-Volterra systems, see \cite{Lotka1925, Volterra1927}. Usually, the Lotka-Volterra systems describe situations of competition between two species, and Kolmogorov proposed a generalization of this kind of system.  We can define a planar Kolmogorov differential system in the following way.
\begin{equation}\label{eqkolp}
	\begin{array}{lll}
		\dot{x}&=&x\,X_m(x,y,\lambda),\\
		\dot{y}&=&y\;Y_m(x,y,\lambda),
	\end{array}
\end{equation} 
where $X_m$ and $Y_m$ are polynomials of degree $m$. In particular when $m=1$, we have the Lotka-Volterra systems. Moreover, this class of systems has a wide range of applications, such as chemical reactions \cite{Her1990}, economics \cite{Gandolfo2008,Goodwin1967,SoRic2002} and hydrodynamics \cite{Bus1981}. Therefore in this work,  piecewise smooth Kolmogorov systems can be described by
	\begin{equation}\label{eqkolpdesc}
		\begin{aligned}
			\begin{array}{l}
				\dot{x}=x\,X_m^{+}(x,y,\lambda^{+}), \; \; \mbox{for} \; y \geq 0, \\
				\dot{y}=y\;Y_m^{+}(x,y,\lambda^{+}),
			\end{array}  \; \; \; \; \; \; \;  \; \; 
			\begin{array}{l}
				\dot{x}=x\,X_m^{-}(x,y,\lambda^{-}), \; \; \mbox{for} \; y \leq 0, \\
				\dot{y}=y\;Y_m^{-}(x,y,\lambda^{-}),
			\end{array}  \\
		\end{aligned}
	\end{equation}
	where $X_m^{\pm}$ and $Y_m^{\pm}$ are polynomials of degree $m$ in the variables $(x,y)$, and $(\lambda^{+},\lambda^{-} ) \in \mathbb{R}^{2m^2+6m+4}$ are perturbative parameters. Even though in biological applications, the Kolmogorov system is usually considered in the first quadrant, in this work we will make a change of coordinates to translate the equilibrium point of center type from the first quadrant to the origin, facilitating computations and this justifying the choice of such discontinuity line $\Sigma= \{ (x,y)  \in \mathbb{R}^{2} \; | \; y=0\}$. Kolmogorov system of form \eqref{eqkolpdesc} can be viewed as a truncation of a number of biologic competing models considered in the recent literature \cite{Joydeb, Cortes, Luo}. which work with analytic $X_m^\pm $ and $Y_m^\pm$. Nullifying thermal variation and freshwater flux parameters in the Stommel-box model \cite{AlkAshJacQui2019,BudGriKus2022} one obtains a non-smooth system of form  \eqref{eqkolpdesc}. Nullifying one of the efficacies to inhibit prostate cancer cells proliferation or eliminating the irreversible mutation, the model of prostate cancer  \cite{yan2022dynamics} takes the form of Kolmogorov system \eqref{eqkolpdesc} too.

	Usually, for polynomial vector fields of degree $n$, studying the maximum number of limit cycles
	$\mathcal H(n)$ (called by Hilbert number)  in terms of their degrees is an open problem and approached in the second part of the 16th Hilbert Problem. On the other hand, there is a local version of the 16th Hilbert Problem, where $\mathcal M(n)$ is denoted as the maximum number of limit cycles bifurcating from a monodromic equilibrium point and, 
	$\mathcal M(n) \leq \mathcal H(n)$. It is well known that linear systems do not have limit cycles then $\mathcal M(1)=\mathcal H(1)=0$ and for other degrees, we can find a good description of the respective quotas of limit cycles in \cite{Bautin1954,ProTor2019}. For piecewise smooth polynomial vector fields of degree $n$, the respective  numbers are denoted by $\mathcal M^{c}_p(n)$ and $ \mathcal H^{c}_p(n)$, where the upper index $c$ means only crossing limit cycles are counted. 
	
	For planar piecewise smooth polynomial vector fields with two zones separated by a straight line, results on the number of of crossing limit cycles are obtained in \cite{Goutor,GuoYuChe2018}.
	 Moreover, the authors in \cite{GouvTorre2021} have given new lower bounds for the local cyclicity of piecewise smooth polynomial vector fields, $\mathcal M^{c}_p(n)$, for $n=3,4,5$.
	
	Regarding planar Kolmogorov systems \eqref{eqkolp} of degree $n=m+1$, the classical theory defines $\mathcal M_{K}(n)$ as the maximal number of small limit cycles bifurcating from an equilibrium point of the system \eqref{eqkolp}. In this context, the 16th Hilbert Problem is even more restricted, because perturbative parameters must keep the system in the Kolmogorov form and consequently $\mathcal M_{K}(n) \leq \mathcal M(n)$. As expected, for $n=1$ we have $\mathcal M_{K}(1)=0$. Bautin in \cite{Bautin1952} showed that $\mathcal M_{K}(2)=0$. In \cite{LloydPearson1996}, the authors came up with the first example of a cubic Kolmogorov system with six limit cycles. Recently the authors in \cite{kol} showed a new example of a cubic Kolmogorov system with six limit cycles, that is, $\mathcal M_{K}(3) \geq 6$. They also provided the first quartic Kolmogorov system with thirteen limit cycles, $\mathcal M_{K}(4) \geq 13$, and the first quintic Kolmogorov system with twenty-two limit cycles, $\mathcal M_{K}(5) \geq 22$.

	As far as we know, there are still no works aimed at studying sewing limit cycles for piecewise smooth Kolmogorov system. We will consider local cyclicity $\mathcal M^{c}_{p_K}(n)$ as the maximal number of small limit cycles, all nested and surrounding one equilibrium point or a sliding segment of the system \eqref{eqkolpdesc}. As in the analytical case, it is expected $\mathcal M^{c}_{p_K}(n) \leq \mathcal M^{c}_{p}(n)$. Our first objective consists in verifying that $\mathcal M^{c}_{p_K}(n)$ is still null for $n=1,2$, which happens in the classical theory for the Kolmogorov systems. 
	
	For $n=1$ we do not have crossing limit cycles in \eqref{eqkolpdesc}, because the discontinuity line does not have crossing regions, that is,  $\mathcal M^{c}_{p_K}(1)=0$. The same happens if we consider the straight line $\{x=0\}$. This fact is even more general, i.e., we can consider the line $\{y-a \, x=0 \}$, and then the Kolmogorov vector fields will always point to the same side of the straight line, which makes it impossible to create a crossing limit cycle. After that,  we show the existence of at least one crossing limit cycle in Palomba's economics model, considering it from a piecewise smooth point of view \eqref{eqkolpdesc} for $m=1$, that is, $n=2$ and $\mathcal M^{c}_{p_K}(2) \geq 1$. We also provide lower bounds of local cyclicity for the piecewise smooth Kolmogorov system for $n=3$ and $n=4$.



	%
	%
	%
	
	%
	
	Our main result in this work is given below.
	
	\begin{teo}\label{teo}
		The lower bounds of small limit cycles $\mathcal M^{c}_{p_K}(n)$ for the piecewise smooth Kolmogorov systems \eqref{eqkolpdesc} for $n=2,3,$ and $4$ are, at least, $1$, $12$, and $18$, respectively (as illustrated in Table 1).
	\end{teo}

%
%
%
%
%
%
	
		\begin{table}[h]
		\begin{center}
			\begin{tabular}{|l|c|c|}
				\cline{2-3}
				\hline
				\multicolumn{1}{|c|}{deg} & Kolmogorov & Piecewise smooth Kolmogorov  \\
				\hline  
			 
				$n=1$   & $\mathcal M_{K}(1)=0$ & $\mathcal M^{c}_{p_K}(1)=0,$  \\
				
				$n=2$ &  $\mathcal M_{K}(2)=0$& $\mathcal M^{c}_{p_K}(2) \geq 1,$   \\
				
				$n=3$ &$\mathcal M_{K}(3) \geq 6$& $\mathcal M^{c}_{p_K}(3) \geq 12,$   \\
				
				$n=4$  &$\mathcal M_{K}(4) \geq 13$& $\mathcal M^{c}_{p_K}(4)\geq 18.$  \\ 
				\hline
			\end{tabular}
		\end{center}
		
		\vspace{0.2cm}
		
		\caption{Summary of local Hilbert numbers for Kolmogorov systems of degree $1\leq n \leq 4$ for smooth and piecewise smooth cases.}
	\end{table}
	
	This paper is structured as follows. In Section~\ref{se:preliminaries} we present a method to compute Lyapunov constants. In addition, we show some results that allow us to obtain limit cycles using the Taylor expansions of the Lyapunov constants. In Section~\ref{secpalomba} we show the existence of at least one crossing limit cycle in Palomba's economics model, considering it from a piecewise smooth point of view, i.e. viewing it as a piecewise smooth Kolmogorov system of degree $n=2$. Finally, in the Section~\ref{se:bifurcation} we prove the remaining cases of our  main result.


\section{Preliminaries}\label{se:preliminaries}

In this section, we recall some basic concepts of piecewise smooth differential systems and how to obtain the coefficients of the Taylor series of the return map, the so-called Lyapunov constants. Let us consider the following piecewise planar differential system
	\begin{equation}\label{eqkolpg}
	\begin{aligned}
		\begin{array}{l}
			\dot{x}=X^{+}(x,y,\lambda^{+}), \; \; \mbox{for} \; y \geq 0, \\
			\dot{y}=Y^{+}(x,y,\lambda^{+}),
		\end{array}  \; \; \; \; \; \; \;  \; \; 
		\begin{array}{l}
			\dot{x}=X^{-}(x,y,\lambda^{-}), \; \; \mbox{for} \; y \leq 0, \\
			\dot{y}=Y^{-}(x,y,\lambda^{-}).
		\end{array}  \\
	\end{aligned}
\end{equation}
If we consider $Z^{\pm}=(X^{\pm},Y^{\pm})$ then following the concepts introduced by Fillipov, see \cite{Fil1988}, we can have on $\Sigma= \{ y=0\}$ three different behaviors that we will call crossing, escaping, and sliding ($\Sigma^{C}$, $\Sigma^{S}$, $\Sigma^{E}$). These concepts are valid in general for a curve that is the inverse image of a regular value. Considering $f(x,y)=y$, we say that $p \in \Sigma^{C}$ if, and only if $Z^{+}f(p)\cdot Z^{-}f(p)>0$, where $Z^{\pm}f(p)=\langle \nabla f(p),Z^{\pm}f(p)\rangle$. On the other hand, we have $p \in \Sigma^{D}\bigcup\Sigma^{S}$ if, and only if $Z^{+}f(p)\cdot Z^{-}f(p)<0$, see Figure~\ref{fig1}.

\begin{figure}[h]\vspace{0.3cm}
	\begin{overpic}[width=10.0cm]{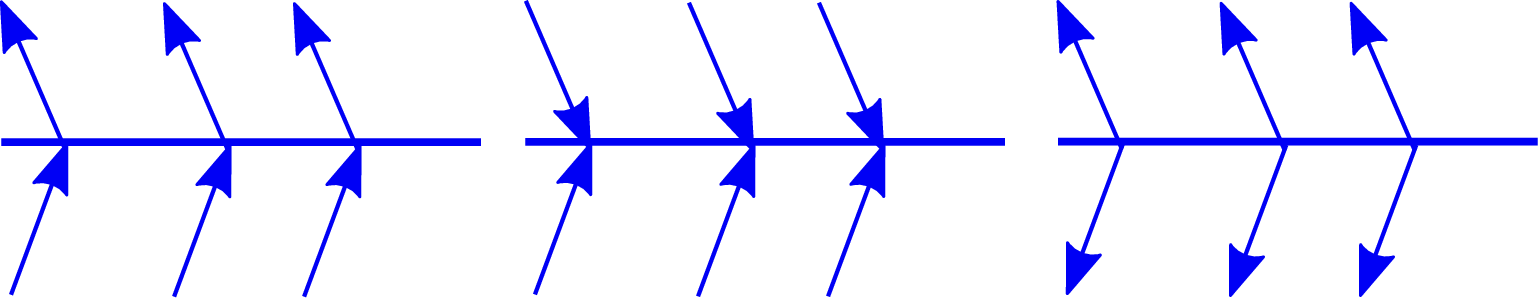}
	\end{overpic}
	\caption{Crossing, Sliding and escaping segments.}
	\label{fig1}
\end{figure}


When the origin of the system \eqref{eqkolpg} is monodromic, we can study the return map. Since we have two vector fields ($Z^{\pm}$), we will have two half-return maps, that we call by $\Pi^{\pm}$, and their compositions provide us with the global return map $\Pi(\rho)=\Pi^{-}(\Pi^{+}(\rho))$. Analogously to the classic case, the zeros of the return map correspond to isolated periodic orbits. Equivalently, we can obtain periodic orbits finding zeros of the displacement map $\Pi(\rho)-\rho$ or the difference map $\Delta(\rho)=(\Pi^{-})^{-1}(\rho)-\Pi^{+}(\rho)$, see Figure~\ref{fig2}. The coefficients of the Taylor series of the difference map $\Delta$ at the origin are also called Lyapunov constants for piecewise smooth polynomial vector fields.
\begin{figure}[h]\vspace{0.3cm}
	\begin{overpic}[width=10.0cm]{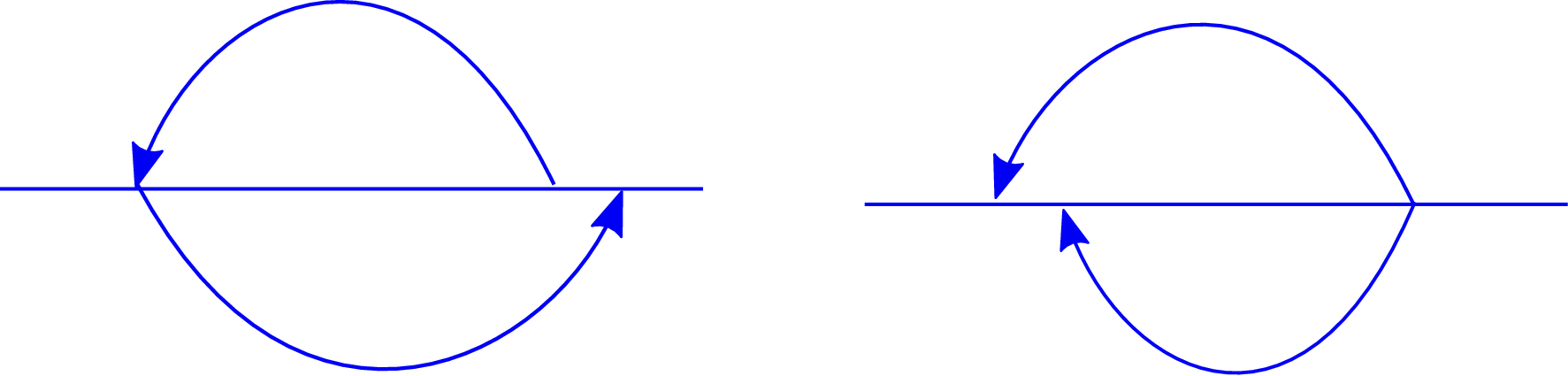}
	\end{overpic}
	\caption{Crossing, Sliding and escaping segments.}
	\label{fig2}
\end{figure}

Since we will be dealing with the nondegenerate center (i.e. the linear part of the center has eigenvalues purely imaginary), and the perturbative parameters will be different on $y>0$ and $y<0$, we can write system~\eqref{eqkolpg} in the following way

\begin{equation}\label{eqfn}
\begin{array}{lll}
	\dot{x} & = &-y+P^{\pm}(x,y,\lambda^{\pm}),\\
	\dot{y} & = & \phantom{-}x+Q^{\pm}(x,y,\lambda^{\pm}),
	\end{array}
\end{equation} 
where $P^{\pm}$, $Q^{\pm}$ are polynomials without constant or linear terms. With this hypothesis the trace of the Jacobian matrix of \eqref{eqfn} at the origin is zero and the piecewise smooth linearized system does not have sliding or escaping segments. Passing system~\eqref{eqfn} to polar coordinates $(r,\theta)$ and leaving $\theta$ as the new independent variable, we get the following analytic differential equations
\begin{equation}\label{eqfn1}
	\begin{array}{lll}
		\dfrac{dr}{d\theta} & = &\displaystyle\sum_{i=2}^{\infty}R^{+}_{i}(\theta,\lambda^{\pm}) \,  r^{i}, \mbox{ $\theta \in [0,\pi]$}\\
		\dfrac{dr}{d\theta} & = &\displaystyle\sum_{i=2}^{\infty}R^{-}_{i} (\theta,\lambda^{\pm}) \,  r^{i}, \mbox{ $\theta \in [\pi,2\pi]$},
	\end{array}
\end{equation} 
where $R^{\pm}_{i}(\theta,\lambda^{\pm})$ are trigonometrical polynomials in the variable $\theta$.  If we consider the initial condition $r^{\pm}(0,\rho,\lambda^{\pm})=\rho$ then in a neighborhood of the origin we can write the solutions of \eqref{eqfn1} with these initial conditions as
\begin{equation*}
r^{\pm}(\theta,\rho,\lambda^{\pm})=\rho+\displaystyle\sum_{i=2}^{\infty}r_{i}^{\pm}(\theta,\lambda^{\pm}) \, \rho^{i},
\end{equation*} 
%
%
which is also analytic. By plugging it into the respective equation \eqref{eqfn1} we can find each $r_{i}^{\pm}(\theta,\lambda^{\pm}), \; i \geq 2$,  by solving simple differential equations with initial conditions $r_{i}^{\pm}(0,\lambda^{\pm})=0$. Therefore, as our piecewise smooth systems \eqref{eqfn} are defined separated by the straight line $\{y = 0\}$, we can define the half-return maps close to the origin by
\begin{equation*}
	\begin{array}{lll}
	\Pi^{+}(\rho) & = &\rho +\displaystyle\sum_{i=2}^{\infty}r_{i}^{+}(\pi,\lambda^{+}) \, \rho^{i},\\
	(\Pi^{-})^{-1}(\rho) & = &\rho +\displaystyle\sum_{i=2}^{\infty}r_{i}^{-}(-\pi,\lambda^{-}) \, \rho^{i},
	\end{array}
\end{equation*} 
and finally,
\begin{equation}\label{displacement}
	\Delta(\rho)=\displaystyle\sum_{i=2}^{\infty}(r_{i}^{-}(-\pi,\lambda^{-})-r_{i}^{+}(\pi,\lambda^{+})) \, \rho^{i}=\displaystyle\sum_{i=2}^{\infty}L_{i} \, \rho^{i}.\\
\end{equation} 
As the systems \eqref{eqfn} are analytic, follow that the half-return maps $\Pi^{\pm}$, and $\Delta$  are also analytic.


 We call the coefficients $L_{i}$ in \eqref{displacement} as the generalized  Lyapunov constants associated with system \eqref{eqfn}.  As in the classic case, the first non-vanishing $L_{k}$ provides the stability of the origin and we will say that the origin is a generalized weak-focus of order $k$, that is, $\Delta(\rho)=L_{k} \, \rho^{k} + O(p^{k+1})$ in the piecewise smooth context. We highlight that the above procedure follows closely the classical Lyapunov algorithm described, for example, in \cite{AndLeoGorMai1973} for analytic vector fields. However, while in the classical case, the Lyapunov constants with even indices do not influence the number of limit cycles that can arise, in the piecewise smooth environment it may influence, see \cite[Remark 3.3]{GouvTorre2021}. Therefore, it is expected that we have at least twice as many limit cycles in piecewise smooth systems, once the Lyapunov constants of even indices can also be considered. More details and properties about Lyapunov constants in the classic case can be found in the works \cite{BlowLlo1984}.

\subsection{Bifurcations of Small Limit Cycles}\label{BLC}

Even though in classical Hopf bifurcation for analytic vector fields we obtain at most one limit cycle bifurcating from a weak focus of order one located at the origin, in the piecewise smooth setting we can obtain two limit cycles. The first one comes from the classical Hopf bifurcation and the second one rises from the sliding segment, see \cite{GouvTorre2021}. We summarize this fact in the next result that was proved in \cite{Goutor}.

\begin{pro}\label{prop_2lc} Consider the perturbed system
		\begin{equation*}\label{eqkolpg1}
		\begin{aligned}
			\begin{array}{l}
				\dot{x}=-(1+c^2)y+\displaystyle\sum_{k+l=2}^{\infty}a_{kl}^{+}x^{k}y^{l}, \;  \; \mbox{for} \; y \geq 0, \\
				\dot{y}=\phantom{-} x+2cy+\displaystyle\sum_{k+l=2}^{\infty}b_{kl}^{+}x^{k}y^{l},
			\end{array}  \; \; \; \;
			\begin{array}{l}
				\dot{x}=-y+\displaystyle\sum_{k+l=2}^{\infty}a_{kl}^{-}x^{k}y^{l}, \; \; \mbox{for} \; y \leq 0, \\
				\dot{y}= \phantom{-} d+x+\displaystyle\sum_{k+l=2}^{\infty}b_{kl}^{-}x^{k}y^{l}.
			\end{array}  \\
		\end{aligned}
	\end{equation*}
	If $a_{11}^{+}+2 \, b^{+}_{02}+b_{20}^{+} \neq a_{11}^{-}+2 \, b_{02}^{-}+b_{20}^{-}$ then there exist $c$ and $d$ arbitrarily small such that two crossing limit cycles of small amplitude bifurcate from the origin.
\end{pro}

\begin{remark}
We can write the displacement map 
\begin{equation*}
	\Delta(\rho)=L_{0}(c,d)+L_{1}(c,d) \, \rho+L_{2} \, \rho^2+\ldots,
\end{equation*}	
 where $L_{0}$ and $L_{1}$ are polynomials in $c$ and $d$. If $d=0$ then there is no sliding segment, and we have the classical Hopf bifurcation. On the other words, only one crossing limit cycle of small amplitude bifurcates when $c$ is small enough. The classical Hopf bifurcation, see \cite{MarMcC1976}, occurs when a complex conjugate pair of eigenvalues (with real part not null) of the linearised system at an equilibrium point becomes purely imaginary. Then the birth of a limit cycle from a weak focus of first-order arises varying the trace of the matrix of the linearized system from zero to a number small enough. In this way, we will denote by $L_1$ the trace of the perturbed system, clearly $L_0(0,0)=0$ for the unperturbed one. When $d\neq0$, inside the limit cycle of small amplitude, there is only a sliding segment. For the appropriate value of $d$, a second bifurcation arises. This is called pseudo-Hopf bifurcation, see more details in \cite{CasLliVer2017}. Then, using this argument we will show that $\mathcal M^{c}_{p_K}(2)\geq 1$. We observe that for the continuous case, $\mathcal M_{K}(2)=0$. However, in the non-smooth case, we have, at least, one crossing limit cycles.
\end{remark}

The degenerate Hopf bifurcation is a natural generalization of this bifurcation phenomenon when $k$ small limit cycles appear from a weak focus of order $k$. In general, the complete unfolding of $k$ limit cycles near a weak focus of order $k$ is only guaranteed when perturbations are analytic, see for example \cite{Rou1998}. When the perturbation is restricted to be a polynomial of some fixed degree this property is not automatic. Because of this, the problem of finding the cyclicity of a center is so difficult. A way to avoid these difficulties is to study lower bounds for cyclicity. Christopher, see \cite{Chr2006}, provides the necessary conditions to obtain lower bounds for the cyclicity of a center. The author details how we can use the first-order Taylor approximation of the Lyapunov constants to obtain lower bounds on the number of limit cycles near center-type equilibrium points. The author details how we can use the first-order Taylor approximation of the Lyapunov constants to obtain lower bounds on the number of limit cycles near center-type equilibrium points. Similar results can be found in \cite{ChiJac1989,Han1999}.
	
	\begin{pro}[\cite{Chr2006}]\label{thm:chrorder1}
	Consider the perturberd piecewise smooth polynomial differential system of degree $n$ of the form \eqref{eqfn}, with $P_{c},Q_{c}$ without linear or constant terms and the unperturbed vector field, $(\dot{x},\dot{y})=(-y+P_{c}(x,y),x+Q_{c}(x,y))$, has a center at the origin. If the first-order Taylor developments with respect to the perturbation parameters at the origin of the first $k-1$ Lyapunov Constants are linearly independent, then there exist perturbation parameters $(a_{kl}^{\pm},b_{kl}^{\pm})$ such that system
			\begin{equation*}\label{systemperturbated}
		\begin{aligned}
			\begin{array}{l}
				\dot{x}=-y+P_{c}(x,y)+\displaystyle\sum_{k+l=2}^{n}a_{kl}^{+}x^{k}y^{l}, \;   \mbox{if} \, y \geq 0, \\
				\dot{y}=\phantom{-} x+Q_{c}(x,y)+\displaystyle\sum_{k+l=2}^{n}b_{kl}^{+}x^{k}y^{l},
			\end{array}  
			\begin{array}{l}
				\dot{x}=-y+ P_{c}(x,y) +\displaystyle\sum_{k+l=2}^{n}a_{kl}^{-}x^{k}y^{l}, \; \mbox{if} \, y \leq 0, \\
				\dot{y}=\phantom{-} x+Q_{c}(x,y)+\displaystyle\sum_{k+l=2}^{n}b_{kl}^{-}x^{k}y^{l},
			\end{array}  \\
		\end{aligned}
	\end{equation*}
has, at least, k crossing limit cycles of small amplitude bifurcating from the origin.	
\end{pro}

\section{An application to Palomba's model}\label{secpalomba}

We dedicate this section to showing a practical example of a piecewise smooth Kolmogorov system when $m=1$ with, at least, one crossing limit cycle surrounding the sliding segment. 

In 1939, Palomba considered an economic model where only two types of goods exist, consumption goods $\textbf{(a)}$ and capital goods $\textbf{(b)}$, see \cite{Gandolfo2008, Palomba}. The goods of type $\textbf{(a)}$ are ready for immediate consumption. Moreover, goods of type $\textbf{(b)}$ are used to produce new goods. Palomba made some assumptions. 
%
%
%
%
%
%
%
First of all, the economy is in a dynamic situation tending to increase its capital equipment. Then part of the goods of type $\textbf{(a)}$  are allocated to goods of type $\textbf{(b)}$. Given any instant, goods of type $\textbf{(a)}$ have a coefficient of increase equal $\eta_{2}$, while goods of type $\textbf{(b)}$ have a coefficient of increase equal to $-\eta_{1}$.  This means that, in the absence of the allocation described above, goods of type $\textbf{(a)}$ would increase continuously, while goods of type $\textbf{(b)}$ would decrease toward zero. On the other hand,  $-\gamma_{2}$ is the coefficient of decrease of the goods of type $\textbf{(a)}$,  due to the allocation, while $\gamma_{1}$ is the coefficient of increase of the goods of type $\textbf{(b)}$,  due to the same reason. Here,  $\eta_{1}$, $\eta_{2}$, $\gamma_{1}$, and $\gamma_{2}$ assume positive values. We will denote $x$ and $y$ the volume of goods of type $\textbf{(b)}$ and the volume of types $\textbf{(a)}$, respectively. Then, we can describe Palomba's model as follows:

\begin{equation}
	\begin{array}{lll}\label{Palomba}
		\begin{aligned}
			\dot{x}&=& -x(\eta_{1}-\gamma_{1} \, y),\\
			\dot{y}&=& y(\eta_{2}-\gamma_{2} \, x).
		\end{aligned}
	\end{array}
\end{equation}
The equilibrium points of the system \eqref{Palomba} are $q=(x_{1},y_{1})=(0,0)$ and $p=(x_{2},y_{2})=\left(\eta_{2}/\gamma_{2},\eta_{1}/\gamma_{1}\right)$. In addition, system~\eqref{Palomba} has integrating factor $V=\frac{1}{xy}$ and the system has a center around the point $p$.


\begin{figure}[h]\vspace{0.3cm}
	\begin{overpic}[width=5.5cm]
		{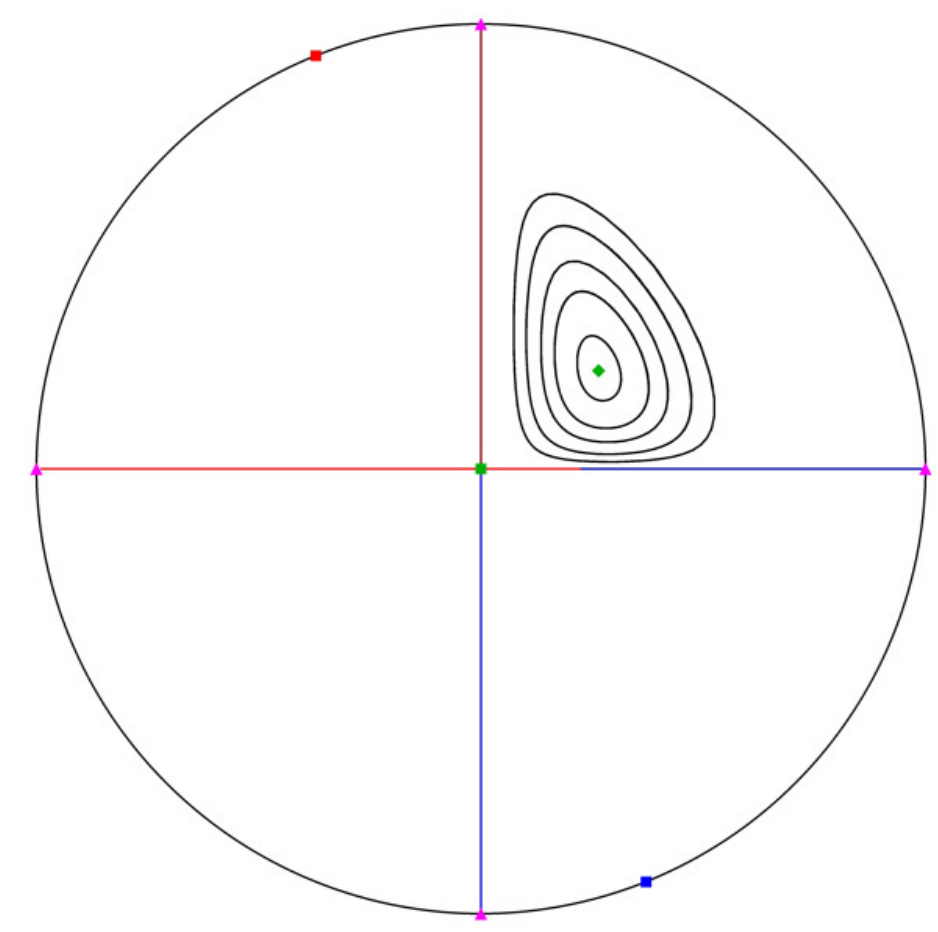}
	\end{overpic}
	\caption{ Phase portrait of the Palomba's model considering $\eta_{1}=1$, $\gamma_{1}=2$, $\eta_{2}=3$, and $\gamma_{2}=5$.}
\end{figure}

When dealing with a dynamic system where two or more variables can affect the growth of one from another, it is common to consider two types of models. For example, considering two species of animals, each species can grow indefinitely and we call this model of exponential growth. However, there can exist factors that impossibility growth at a constant rate, for example, scarcity of resources. In this model, called the logistical model, there are no unlimited resources, in other words, there exists a carrying capacity for the growth. 


%

%

Our goal in this section is to show the existence of one limit cycle in piecewise Palomba's system. Let us consider $\Sigma$ as a parallel segment to the $x-axi$s passing through the center $\left(\eta_{2}/\gamma_{2},\eta_{1}/\gamma_{1}\right)$.  We divide the plane into two half-planes, $\Sigma^{+}$ and $\Sigma^{-}$. In $\Sigma^{-}$ we will consider system~\eqref{Palomba}. In $\Sigma^{+}$,  as there are situations in which resources do not grow continuously, we will add a parameter $\epsilon$ multiplying the resource $y$, thus influencing its growth.

\begin{pro}\label{Palombas_cycle}
Consider the following piecewise smooth Palomba's system
	\begin{equation}\label{Palombadisc}
	\begin{aligned}
		\begin{array}{l}
			\dot{x}=- x \, (\eta_{1}-\gamma_{1} \, y), \; \; \; \; \; \; \; \mbox{for} \; y \geq \eta_{1}/\gamma_{1}, \\
			\dot{y}=\phantom{-}y \, (\eta_{2}+\varepsilon-\gamma_{2} \, x),
		\end{array}  \; \; \; \; \; \; \;  \; \; 
		\begin{array}{l}
			\dot{x}=-x \, (\eta_{1}-\gamma_{1} \, y), \; \; \mbox{for} \; y \leq \eta_{1}/\gamma_{1} \\
			\dot{y}=\phantom{-}y \, (\eta_{2}-\gamma_{2} \, x).
		\end{array}  \\
	\end{aligned}
\end{equation}
Then there exists, at least, one crossing limit cycle surrounding the sliding segment. 
\end{pro}


\begin{proof}

%
%
The parameter $\varepsilon$ determines the sliding segment on $\Sigma$. When $\varepsilon$ varies, the stability of the sliding region changes. This phenomenon is the well-known pseudo-Hopf bifurcation, where one crossing limit cycle arises surrounding the sliding segment. Using Proposition~\ref{prop_2lc}, we can write the displacement function as
 \begin{equation*}
		\Delta{\rho}=L_{0}(c,\varepsilon)+L_{1}(c,\varepsilon) \, \rho+L_{2} \, \rho^2+\ldots,
	\end{equation*}	
where $c=\eta_{1}-\eta_{2}$. We can choose $\varepsilon$ appropriately such that $\Delta{\rho}$ has, at least, one solution $\rho^{*}$. Then we have at least, one crossing limit cycle. The parameter $c$ represents the trace of the Jacobian matrix of the linearized system. From this parameter, occur the classical Hopf bifurcation adding one more crossing limit cycle. However, in this example, when we vanish $c$ and $\varepsilon$, follow that $L_{2}=0$. Then, we can ensure that only the limit cycle from pseudo-Hopf bifurcation arises.
\end{proof}

The Figure~\ref{palombasmodel} shows the pseudo - Hopf Bifurcation for the Palomba's model
	\begin{equation}\label{Palombadisc1}
	\begin{aligned}
		\begin{array}{l}
			\dot{x}=-x \, (1-2 \, y), \; \; \; \; \; \; \; \mbox{for} \; y \geq 1/2, \\
			\dot{y}=\phantom{-}y \, (3 + \epsilon - 5 \, x ),
		\end{array}  \; \; \; \; \; \; \;  \; \; 
		\begin{array}{l}
		\dot{x}=-x \, (1-2 \, y), \; \; \; \; \; \; \; \mbox{for} \; y \leq 1/2, \\
		\dot{y}=\phantom{-}y \, (3  - 5 \, x ).
		\end{array}  \\
	\end{aligned}
\end{equation}
 In addition, to facilitate the computations, we make a change of coordinates to translate point $p$ to the origin. 


\begin{figure}[h]\vspace{0.3cm}\label{palombasmodel}
	\begin{overpic}[width=15.8cm]
		{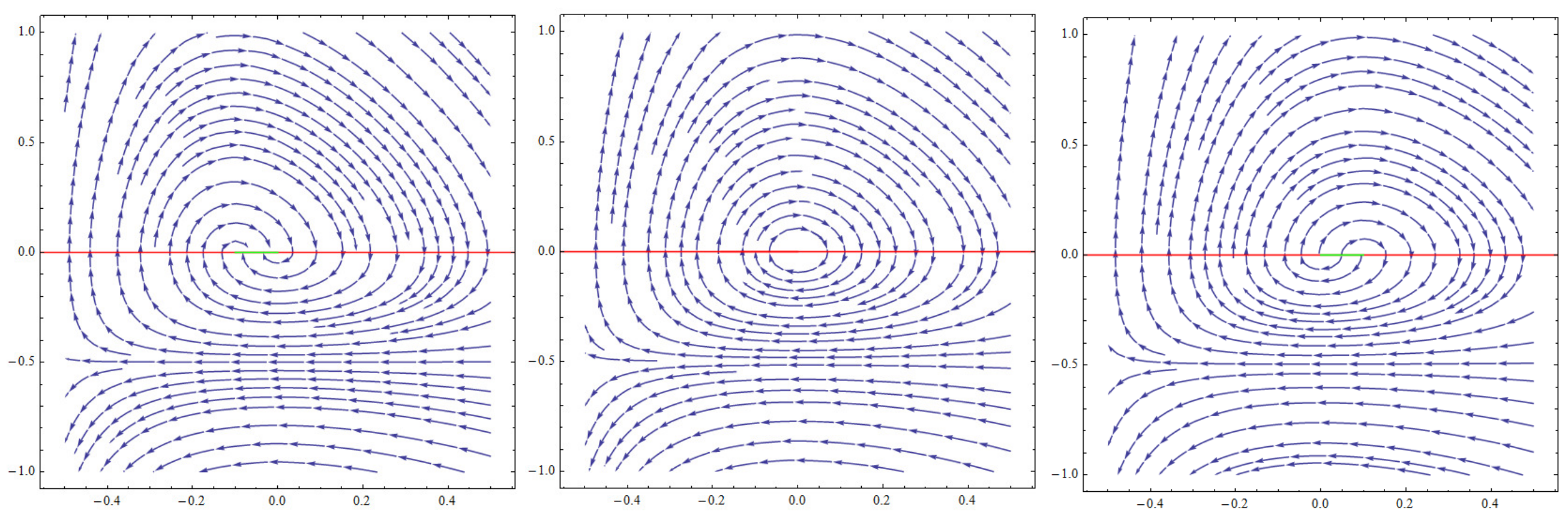}
		\put(15,-3){\textbf{a)}}
		\put(50,-3){\textbf{b)}}
		\put(85,-3){\textbf{c)}}
	\end{overpic}
	\vspace{0.6cm}
	\caption{ Phase portrait of the system~\eqref{Palombadisc1}. The green segment represents the sliding region. In the picture a) we consider $\varepsilon=-0.5$ and the sliding segment is an attractor. In the picture b), $\varepsilon=0$, there is no sliding segment and the point $p$ is a center. In the picture c), $\varepsilon=0.5$ and the sliding segment is repulsor. }
\end{figure}


\section{Bifurcation of limit cycles in Piecewise Smooth Kolmogorov systems}\label{se:bifurcation}	

In this section, we will dedicate to proving Theorem~\ref{teo}. The proof will be an immediate consequence of Proposition~\ref{prop_2lc} and from Propositions~\ref{pro12lc} and \ref{pro18lc} that we will present the following.
	
\begin{pro}\label{pro12lc}
	There exist parameters perturbations such that the center given by 
		\begin{equation}
			\begin{array}{lll}\label{komquar_1}
				\begin{aligned}
					\dot{x}&=& -\frac{x(2x - 8 - y)(y + 1)}{6},\\
					\dot{y}&=&\frac{y(2x - 2 - y)(x - 2)}{3},
				\end{aligned}
			\end{array}
		\end{equation}	
		produces, at least, twelve small limit cycles bifurcating from the center.    
\end{pro}

\begin{proof}
	
The system \eqref{komquar_1} has integrating factor 

\begin{equation*}
	\frac{5(2x + y)}{xy(4x^2 - y^2 - 16x - 2y)},
\end{equation*}
then has a center in the origin. Making a change of coordinates to put the system~(\ref{komquar_1}) in the normal form, we obtain

\begin{equation}
	\begin{array}{lll}\label{komquar_cub}
		\begin{aligned}
			\dot{x}&=& 3x^2y + 1/2yx - 3/2y^2x - y^2 - y,\\
			\dot{y}&=&-6x^2y + 3xy^2 + 2x^2 - 4xy + x,
		\end{aligned}
	\end{array}
\end{equation}	

\begin{figure}[h]\vspace{0.3cm}
	\begin{overpic}[width=5.5cm]
		{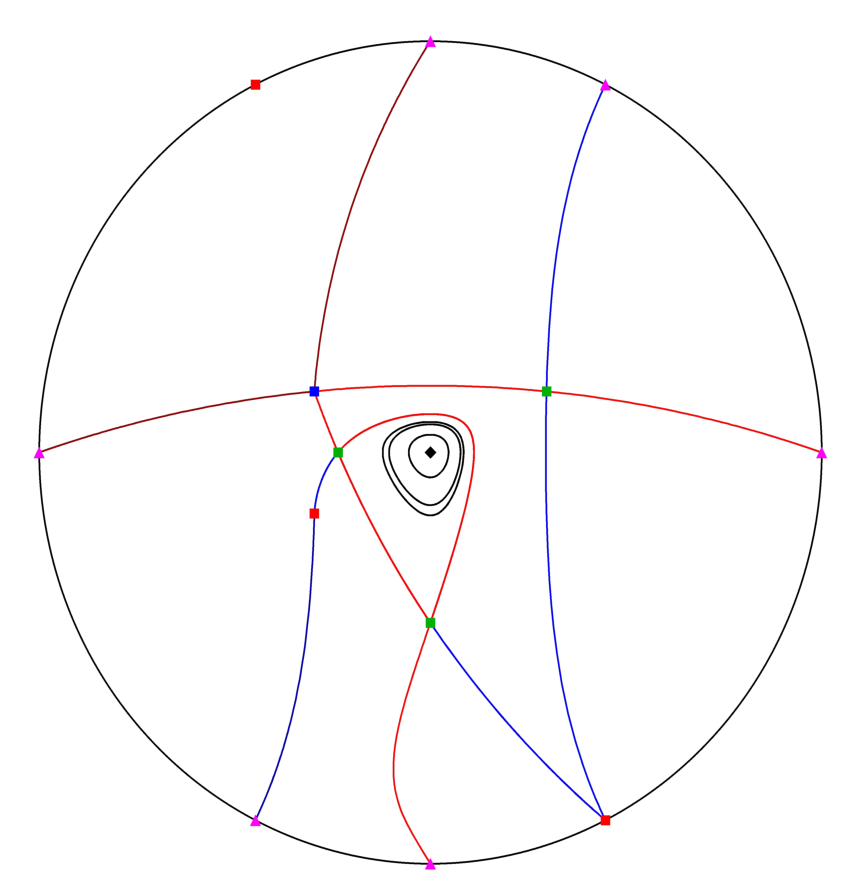}
	\end{overpic}
	\caption{ Phase portrait of the Kolmogorov system~\eqref{komquar_cub}.}
\end{figure}

\begin{figure}[h]\vspace{0.3cm}
	\begin{overpic}[width=5.5cm]
		{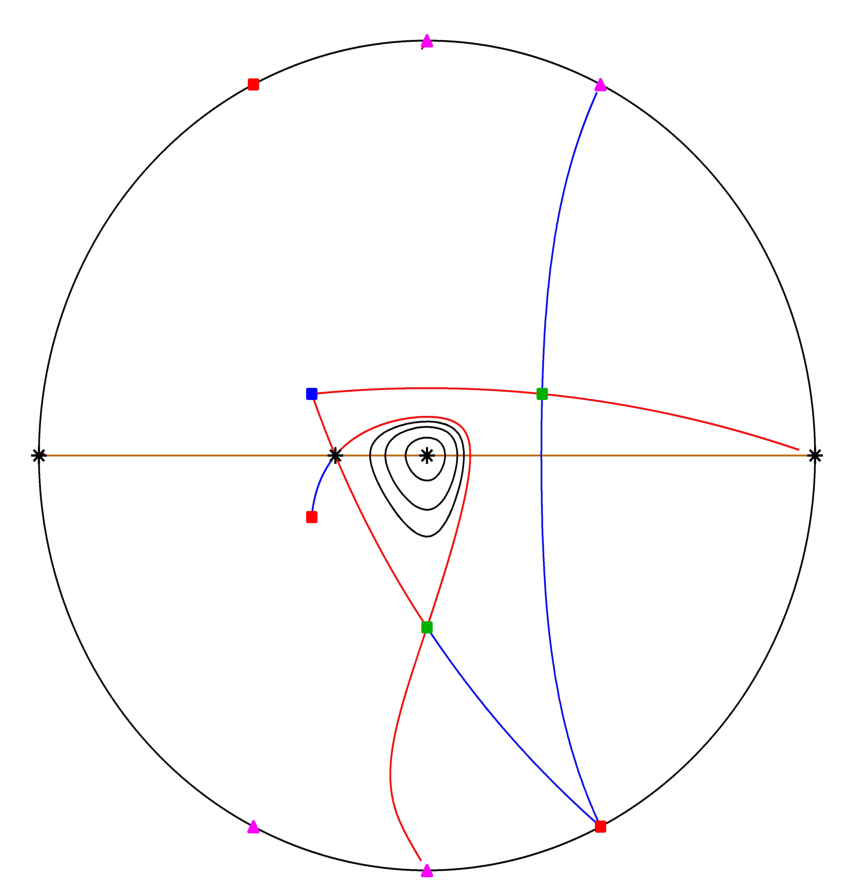}
	\end{overpic}
	\caption{ Phase portrait of the piecewise smooth Kolmogorv system~\eqref{komquar_cub}.}
\end{figure}

Perturbating the system \eqref{komquar_cub} as in \eqref{systemperturbated} with $n=2$ and calculating the linear terms of the first fourteen Lyapunov Constants, we obtain that the rank of the Jacobian Matrix regarding $\Lambda=\left\{a^{\pm}_{20}, a^{\pm}_{11}, a^{\pm}_{02}, b^{\pm}_{20}, b^{\pm}_{1,1}, b^{\pm}_{0,2}, \right\}$ is eleven, with the parameters  $a^{-}_{0, 2}$, $a^{-}_{1, 1}$, $a^{-}_{2, 0}$, $a^{+}_{0, 2}$, $a^{+}_{1, 1}$, $a^{+}_{2, 0}$, $b^{-}_{0, 2}$, $b^{-}_{1, 1}$, $b^{-}_{2, 0}$, $b^{+}_{0, 2}$, $b^{-}_{1, 1}$. Then, it follows from using Proposition~\ref{prop_2lc} 12 small limit cycles bifurcate from the center. Due to the size of the Lyapunov Constants, we will include only the linear terms of the first three Lyapunov Constants.
		\begin{equation*}
			\begin{aligned}
				L_1= & -\frac{2}{3}(-a^{+}_{1, 1} + 2b^{+}_{0, 2} + b^{+}_{2, 0} + a^{-}_{1, 1} - 2*b^{-}_{0, 2} - b^{-}_{2, 0}), \\
				L_2= & \frac{\pi}{8}(2 a^{-}_{0, 2} + a^{-}_{2, 0} + 2 a^{+}_{0, 2} - 2 a^{+}_{1, 1} + a^{+}_{2, 0} + 5 b^{-}_{0, 2} + 2 b^{-}_{1, 1} - 2 b^{-}_{2, 0} + 9 b^{+}_{0, 2} + 2 b^{+}_{1, 1}), \\
				L_{3} = &\frac{8}{45}(11a^{-}_{0, 2} + 10a^{-}_{1, 1} + 4a^{-}_{2, 0} - 11a^{+}_{0, 2} - 10a^{+}_{1, 1} - 4a^{+}_{2, 0} - 3b^{-}_{0, 2} - b^{-}_{1, 1} + 3b^{+}_{0, 2} + b^{+}_{1, 1}).
			\end{aligned}		
		\end{equation*}
\end{proof}
	
\begin{pro}\label{pro18lc} There exist parameters perturbations such that the center given by 
	\begin{equation}
		\begin{array}{lll}\label{komquar_4}
			\begin{aligned}
				\dot{x}&=& -x(y - 1)(29x^2 - 40xy + 40y^2 + 162x + 140y),\\
				\dot{y}&=&y(x - 1)(29x^2 - 40xy + 40y^2 + 162x + 140y),
			\end{aligned}
		\end{array}
	\end{equation}	
	produces, at least, eighteen small limit cycles bifurcating from the center.    
\end{pro}

\begin{proof}
	Putting the system~\eqref{komquar_4} in normal form, we obtain the following system 
	
	\begin{equation*}
		\begin{array}{lll}\label{komquar_5}
			\begin{aligned}
				\dot{x}&=&-\frac{1}{331}y(x + 1)(29x^2 - 40xy + 40y^2 + 180x + 180y + 331),\\
				\dot{y}&=&\frac{1}{331}x(y + 1)(29x^2 - 40xy + 40y^2 + 180x + 180y + 331).
			\end{aligned}
		\end{array}
	\end{equation*}	

Observe that the system
\begin{equation*}
	\begin{array}{lll}\label{komquar_center1}
		\begin{aligned}
			\dot{x}=& x(y - 1), \; \; \; \; \;
			\dot{y}=&y(x - 1),
		\end{aligned}
	\end{array}
\end{equation*}	
has integrating factor $V=\frac{1}{xy}$. It then follows that system~\eqref{komquar_center1} has a center.  As system \eqref{komquar_4} is system \eqref{komquar_center1} multiplied by a quadratic curve of equilibrium points, it follows that system \eqref{komquar_4} has a center. 

\begin{figure}[h]\vspace{0.3cm}
	\begin{overpic}[width=5.5cm]
		{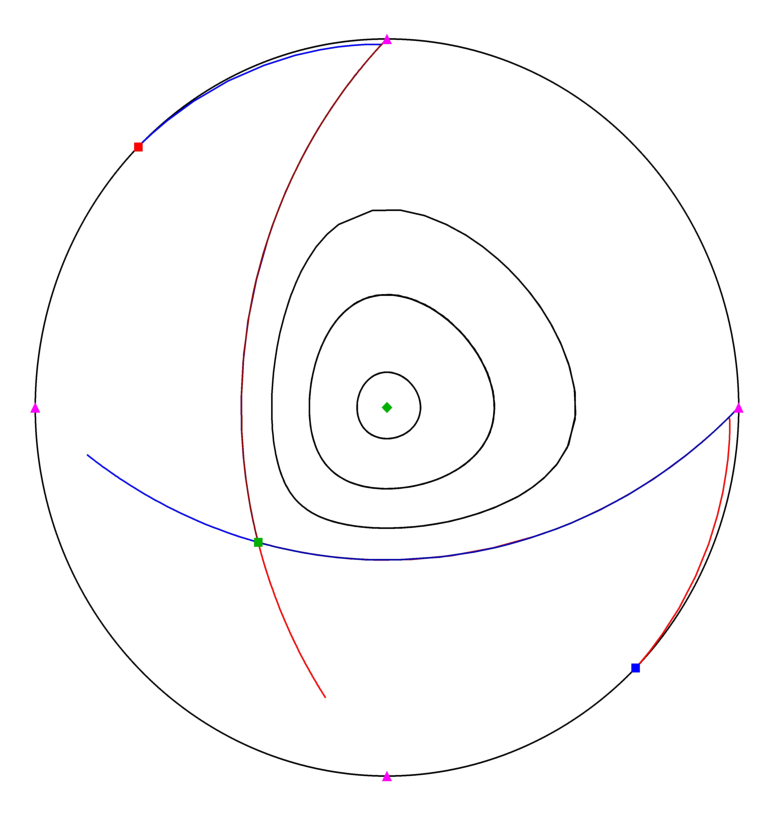}
	\end{overpic}
	\caption{ Phase portrait of Kolmogorov system~\eqref{komquar_4}.}
\end{figure}

\begin{figure}[h]\vspace{0.3cm}
	\begin{overpic}[width=5.5cm]
		{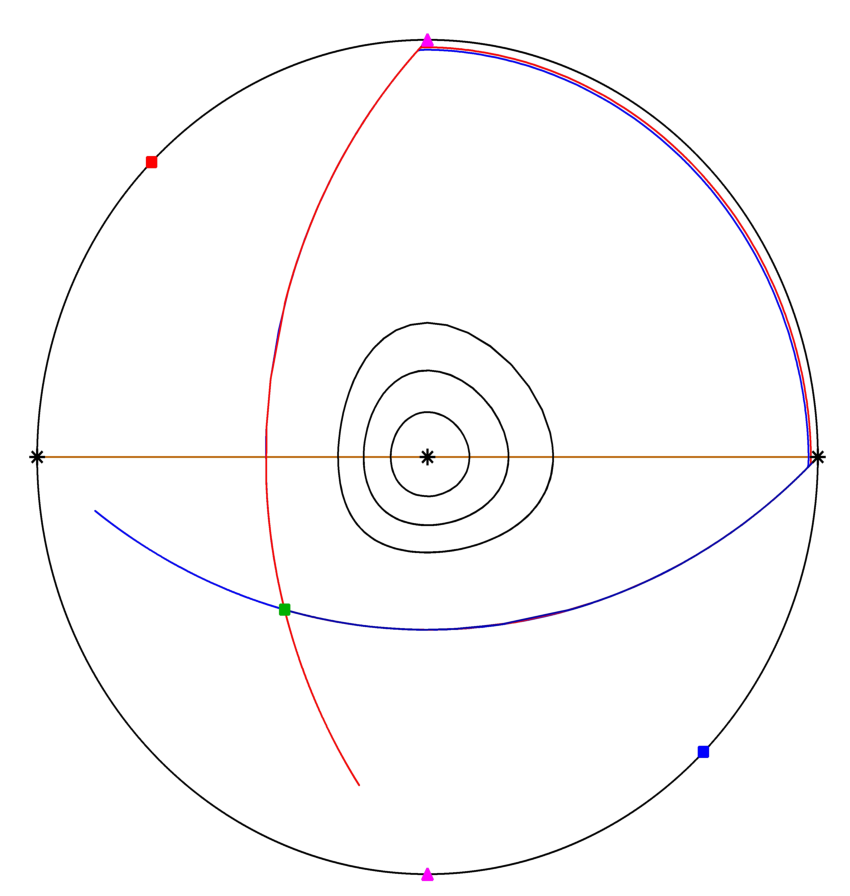}
	\end{overpic}
	\caption{ Phase portrait of the piecewise smooth Kolmogorov system~\eqref{komquar_4}.}
\end{figure}

Following the same scheme as in the previous preposition, calculating the linear terms of the first fourteen Lyapunov Constants, we obtain the rank of the Jacobian Matrix regarding parameters \[\Lambda=\left\{a^{\pm}_{30},a^{\pm}_{21},a^{\pm}_{12},a^{\pm}_{03},a^{\pm}_{20}, a^{\pm}_{11}, a^{\pm}_{02}, b^{\pm}_{30},b^{\pm}_{21},b^{\pm}_{12},b^{\pm}_{03},b^{\pm}_{20}, b^{\pm}_{1,1}, b^{\pm}_{0,2}, \right\}\] is eleven, with the parameters $a^{-}_{0, 3}$,  $a^{-}_{1, 2}$, $a^{-}_{2, 0}$, $a^{-}_{2, 1}$, $a^{-}_{3, 0}$, $a^{+}_{0, 3}$, $a^{+}_{1, 1}$, $a^{+}_{1, 2}$, $a^{+}_{2, 0}$, $a^{+}_{2, 1}$, $a^{+}_{3, 0}$, $b^{-}_{0, 2}$, $b^{-}_{0, 3}$, $b^{-}_{1, 1}$, $b^{-}_{2, 0}$, $b^{+}_{0, 2}$, $b^{+}_{1, 1}$. Again, using Proposition~\ref{prop_2lc}, it follows that 18 small limit cycles bifurcate from the center. Due to the size of the Lyapunov Constants, we will again show only the linear terms of the first three Lyapunov Constants. 
	
	\begin{equation*}
	\begin{aligned}
		L_1= & -\frac{2}{3}(-a^{+}_{1, 1} + a^{-}_{1, 1} - 2 b^{+}_{0, 2} + 2 b^{-}_{0, 2} - b^{+}_{2, 0} + b^{-}_{2, 0}), \\
		L_2= & \frac{\pi}{2648}(180 a^{-}_{0, 2} + 360 a^{-}_{1, 1} - 331 a^{-}_{1, 2} - 122 a^{-}_{2, 0} - 993 a^{-}_{3, 0} + 180 a^{+}_{0, 2} - 331 a^{+}_{1, 2} - 122 a^{+}_{2, 0} \\ &- 993 a^{+}_{3, 0} + 238 b^{-}_{0, 2} - 993 b^{-}_{0, 3} + 180 b^{-}_{1, 1} + 360 b^{-}_{2, 0} - 331 b^{-}_{2, 1} - 482 b^{+}_{0, 2} - 993 b^{+}_{0, 3}\\ & + 180 b^{+}_{1, 1} - 331 b^{+}_{2, 1}), \\
		L_{3} = &-\frac{2}{4930245} (349080 a^{-}_{0, 2} - 357480 a^{-}_{0, 3} - 138358 a^{-}_{1, 2} + 304498 a^{-}_{2, 0} + 121146 a^{-}_{2, 1}\\ & - 864903 a^{-}_{3, 0} - 349080 a^{+}_{0, 2} + 357480 a^{+}_{0, 3} + 138358 a^{+}_{1, 2} - 304498 a^{+}_{2, 0} - 121146 a^{+}_{2, 1}\\ & + 864903 a^{+}_{3, 0} - 82088 b^{-}_{0, 2} \\ &+ 1199544 b^{-}_{0, 3} + 349080 b^{-}_{1, 1} - 357480 b^{-}_{1, 2} - 138358 b^{-}_{2, 1} + 121146 b^{-}_{3, 0} + 82088 b^{+}_{0, 2}\\ & - 1199544 b^{+}_{0, 3} - 349080 b^{+}_{1, 1} + 357480 b^{+}_{1, 2} + 138358 b^{+}_{2, 1} - 121146 b^{+}_{3, 0}).
	\end{aligned}		
\end{equation*}	
\end{proof}

\begin{remark}
 In the examples shown in this work, we believe it may be possible to obtain more limit cycles using the high-order terms of Lyapunov Constants. However, we are limited by computational power. Even using the parallelization tool, see \cite{GouTor2021_2}, calculating high-order terms of the Lyapunov Constants for piecewise smooth systems demands a high computational cost. Therefore, to advance in lower bounds of limit cycles for piecewise systems, finding a more efficient algorithm is fundamental.
\end{remark}


	\section*{Acknowledgements}
	
	Yagor Romano Carvalho was supported by São Paulo Paulo Research Foundation (FAPESP) grants number 2022/03800-7, 2021/14695-7. Luiz F.S. Gouveia was supported by São Paulo Paulo Research Foundation (FAPESP) grants number 2022/03801-3, 2020/04717-0.

\bibliographystyle{abbrv}
\bibliography{biblio.bib} 
\end{document}